\documentclass[12pt]{article}
\usepackage{amsthm,amsfonts,amssymb,amsmath, amscd, enumerate}
\usepackage{csquotes}
\usepackage[colorlinks=true]{hyperref}
\usepackage[all]{xy}

\newcommand{\Br}{{\mathrm{Br}}}

\renewcommand{\div}{{\mathrm{div}}}
\newcommand{\red}{{\mathrm{red}}}

\newcommand{\Gal}{{\mathrm{Gal}}}

\newcommand{\Hom}{{\mathrm{Hom}}}

\newcommand{\coker}{{\mathrm{coker}}}

\newcommand{\NS}{{\mathrm{NS}}}

\newcommand{\rank}{{\mathrm{rank}}}
\newcommand{\Pic}{\mathrm{Pic}}

\newcommand{\uPic}{\underline{\mathrm{Pic}}}

\newcommand{\zl}{{\ZZ_\ell}}
\newcommand{\ql}{{\QQ_\ell}}

\newcommand{\belln}{[\ell^n]}

\newcommand{\muln}{\mu_{\ell^n}}

\newcommand{\gk}{{G_k}}

\renewcommand{\Im}{{\mathrm{Im}}}

\newcommand{\Res}{{\mathrm{Res}}}

\newcommand{\tor}{{\mathrm{tor}}}

\newcommand{\QQ}{\mathbb{Q}}

\newcommand{\ZZ}{\mathbb{Z}} 
 
\newcommand{\GG}{\mathbb{G}}

\newcommand{\lra}{\longrightarrow}

\newtheorem{thm}{Theorem}[section]

\newtheorem{lem}[thm]{Lemma}

\newtheorem{conj}[thm]{Conjecture}

\theoremstyle{definition}

\theoremstyle{remark}

\newtheorem{fact}[thm]{Fact}

\begin{document}
%------------------------------------------------------

\title{Comparison of arithmetic Brauer groups with geometric Brauer groups
\footnote{The main theorem and its method of proof in this paper were essentially covered in a previous work of Colliot-Th\'el\`ene and Skorobogatov. See the second paragraph of the introduction for more details.}
}
\author{Xinyi Yuan}
\maketitle

\tableofcontents

\section{Introduction}

The Brauer group has always been a very important object in number theory and algebraic geometry, as we can see from the classical class field theory, the application of Tsen's theorem to \'etale cohomology of curves, and the Tate conjecture for divisors. 
The goal of this paper is to compare the arithmetic Brauer group with the geometric Brauer group of smooth projective varieties, both of which are natural invariants arising from consideration of the Tate conjecture. 

After the first submission of this paper to arXiv, the author learnt that a major part of the main theorem of this paper (Theorem \ref{comparison}) was previously proved by 
Colliot-Th\'el\`ene and Skorobogatov in \cite[Thm. 2.1]{CTS1}. They only treated the case of characteristic zero, while our paper treated all characteristics simultaneously. However, our main idea is essentially the same as theirs, and the case of positive characteristics does not bring much extra difficulty. 
Hence, our paper should be viewed as explanation of known results and proofs more than claim of original results and proofs.

\subsection{Tate conjecture and geometric Brauer group}

Let us first recall the Tate conjecture for divisors over finitely generated fields. 

\begin{conj}[Conjecture $T^1(X,\ell)$] \label{Tate conjecture}
Let $X$ be a projective, smooth and geometrically integral variety over a finitely generated field $k$ of characteristic $p\geq 0$, and $\ell\neq p$ be a prime number. Then the cycle class map 
$$
\Pic(X)\otimes_\ZZ \QQ_\ell \lra H^2(X^s, \QQ_\ell(1))^{G_k}
$$
is surjective.
\end{conj}

Here we denote $G_k=\Gal(k^s/k)$ and $X^s=X\times_k k^s$, where $k^s$ is the separable closure of $k$. 

The conjecture is proved for abelian varieties by Tate \cite{Tat1} (over finite fields), Zarhin \cite{Zar1, Zar2} (over positive characteristics) and Faltings \cite{Fal1,Fal2} (over characteristic zero); it is proved for K3 surfaces over characteristic zero by Andr\'e \cite{And} and Tankeev \cite{Tan1, Tan2}, and for K3 surfaces over positive characteristics by Nygaard \cite{Nyg}, Nygaard--Ogus \cite{NO}, Artin--Swinnerton-Dyer \cite{ASD}, Maulik \cite{Mau}, Charles \cite{Cha} and Madapusi-Pera \cite{MP}.

For any regular noetherian scheme $X$, the \emph{(cohomological) Brauer group}
$$\Br(X):=H^2(X, \GG_m)$$
is defined to be the \'etale cohomology group. It is automatically torsion by \cite[Prop. 1.4]{Gro2} since $X$ is regular. 
The Tate conjecture is related to the Brauer group by the following result. 

\begin{thm} \label{equivalence}
Let $X$ be a projective, smooth and geometrically integral variety over a field $k$ of characteristic $p\geq 0$, and $\ell\neq p$ be a prime number. 
Then there is a canonical exact sequence
$$
0\lra \NS(X)\otimes_\ZZ \QQ_\ell \lra H^2(X^s, \QQ_\ell(1))^{G_k}
\lra V_\ell(\Br(X^s)^{G_k}) 
\lra 0.
$$
In particular, if $k$ is finitely generated, then $T^1(X,\ell)$ holds if and only if $\Br(X^s)^{G_k}[\ell^\infty]$ is finite.
\end{thm}

Here for an abelian group $G$, we denote  
$\displaystyle T_\ell(G)=\Hom_{\ZZ}(\QQ_\ell/\ZZ_\ell, G)$ and $\displaystyle V_\ell(G)=T_\ell(G)\otimes_{\ZZ_\ell}\QQ_\ell$ for the Tate modules.

This theorem is known to experts, as a generalization of the work over finite fields of Tate \cite{Tat2}. See \cite[Thm. 1.1]{Tan3}, \cite[Prop. 2.5]{SZ}, and \cite[Prop. 4.1]{CTS2} for the result under some different or extra assumptions. We will write a complete proof of the theorem in \S\ref{sect tate brauer} for convenience of readers. Our proof is a variant of that of \cite[Prop. 2.5]{SZ}, and the key is to the existence of a $G_k$-invariant splitting of the exact sequence 
$$
0\lra \NS(X^s)\otimes_\ZZ \QQ_\ell \lra H^2(X^s, \QQ_\ell(1))
\lra V_\ell\Br(X^s) 
\lra 0.
$$
The splitting implies that the $G_k$-invariants of the sequence form a short exact sequence. 
Such a splitting is constructed as the orthogonal complement of 
$\NS(X^s)\otimes_\ZZ \QQ_\ell$ in $H^2(X^s, \QQ_\ell(1))$
under the cup-product pairing of $H^2(X^s, \QQ_\ell(1))$ with itself (with respect to a fixed ample line bundle over $X$).

\subsection{Geometric Brauer group vs Arithmetic Brauer group}

In Theorem \ref{equivalence},  $\Br(X^s)$ is 
the \emph{geometric Brauer group} of $X$, comparing to the 
\emph{arithmetic Brauer group} $\Br(X)$ of $X$. 
We are interested in the natural map $\Br(X)\to \Br(X^s)^{G_k}$, and particularly how far it is from an isomorphism. 
The main result of this paper is as follows. 

\begin{thm} \label{comparison}
Let $X$ be a projective, smooth and geometrically integral variety over a field $k$ of characteristic $p\geq0$. 
 \begin{itemize}
\item[(1)] There is a canonical injection of $\ker(\Br(X)/\Br(k)\to \Br(X^s)^{G_k})$ into $H^1(k, \Pic(X^s))$, and the cokernel of this injection has a finite exponent. 
\item[(2)] If $p=0$, then the cokernel of 
$\Br(X)/\Br(k) \to \Br(X^s)^{G_k}$ is finite.
\item[(3)] If $p>0$, then the cokernel of 
$\Br(X)/\Br(k) \to \Br(X^s)^{G_k}$ is isomorphic to the direct sum of a finite group of order prime to $p$ and a $p$-group of finite exponent.
\item[(4)] If $k$ is a finite field,  then the kernel and the cokernel of 
$\Br(X)/\Br(k) \to \Br(X^s)^{G_k}$ are both finite.
\end{itemize}
\end{thm}

%Here in (3), the term $G\nonp$ for a torsion abelian group $G$ denotes the subgroup of $G$ of elements of orders prime to $p$.

% By part (2) of the theorem, we see that Question 1 and Question 2 of \cite{SZ} are equivalent in this case. Another quick consequence of the theorem is as follows.

Part (2) of the theorem (for $p=0$) was previously proved by Colliot-Th\'el\`ene and Skorobogatov in \cite[Thm. 2.1]{CTS1}. Our method is essentially the same as that of the loc. cit., where the key idea is a pull-back trick. We will talk about the idea later. 

If $k$ is a finite field, combining Theorem \ref{equivalence} and 
Theorem \ref{comparison}(4), we see that $T^1(X,\ell)$ is equivalent to the finiteness of $\Br(X)[\ell^\infty]$. This is compatible with the pioneering results of \cite[Thm. 5.2]{Tat2} over finite fields.

Another interesting case happens when $H^1(X^s, \QQ_\ell)=0$ for a finitely generated $k$ of characteristic $p\geq0$ and a prime $\ell\neq p$. This implies $\Pic^0(X^s)=0$, and thus $\Pic(X^s)=\NS(X^s)$
is finitely generated. 
Then we have a canonical isomorphism 
$$V_\ell(\Br(X)/\Br(k)) \lra V_\ell(\Br(X^s)^{G_k}).$$
It follows that $T^1(X,\ell)$ is equivalent to the vanishing of the either of these rational Tate modules. 

In general, if $k$ is a finitely generated field of characteristic $p>0$, it is reasonable to conjecture that for any prime $\ell\neq p$, the canonical map 
$$V_\ell(\Br(X)/\Br(k)) \lra V_\ell(\Br(X^s)^{G_k})$$
is an isomorphism. By the theorem, this is equivalent to 
$$V_\ell(H^1(k, \Pic(X^s)))=0,$$
which in turn is equivalent to
$$V_\ell(H^1(k, \Pic^0(X^s)))=0.$$

%\begin{remark}
%Note that  \cite[Thm. 1.6]{Yua} further reduces 
%$T^1(X)$ to projective and smooth surfaces $X$ over $k$ with $H^1(X,\CO_X)=0$. 
%\end{remark}

\subsection{The pull-back trick}

Theorem \ref{comparison} is proved by careful analysis of related connecting maps in the Hochschild--Serre spectral sequence
$$
E_2^{a,b}= H^a(k, H^b(X^s, \GG_m))
\Longrightarrow  H^{a+b}(X, \GG_m).
$$
The spectral sequence induces a canonical long exact sequence 
$$
0\lra \ker(\alpha_X:H^1(k, \Pic(X^s))\to H^3(k,\GG_m))
\lra \Br(X)/\Br(k) \stackrel{ }{\lra} 
E_3^{0,2}\stackrel{\gamma_X}{\lra} E_3^{3,0}.
$$
Here
$$
E_3^{0,2}=\ker(\beta_X: \Br(X^s)^{G_k}\to H^2(k, \Pic(X^s)),
$$
$$
E_3^{3,0}=\coker(\alpha_X:H^1(k, \Pic(X^s))\to H^3(k,\GG_m)).
$$
Then the theorem is a consequence of the fact that the images of all the canonical maps 
$$\alpha_X:H^1(k, \Pic(X^s))\lra H^3(k,\GG_m),$$
$$\beta_X:\Br(X^s)^{G_k} \lra H^2(k, \Pic(X^s)),$$ 
and
$$\gamma_X:E_3^{0,2}\lra E_3^{3,0}$$ 
have finite exponents.  

This is proved in Lemma \ref{connecting map}. 
The images of $\alpha_X$ and $\gamma_X$ are easy to treat, following from the fact that the kernel of $H^3(k,\GG_m)\to H^3(X,\GG_m)$ has a finite exponent. 

The hard part is to bound the image of $\beta_X$.
The key idea is a pull-back trick used in Colliot-Th\'el\`ene and Skorobogatov \cite{CTS1}.

The pull-back trick is based on the observation that if $\dim X=1$, then $\Br(X^s)=0$ and thus the connecting map $\beta_X$ is zero. 
For general $\dim X$, we take suitable curves in $X$ and apply the functoriality of the connecting maps to relate the maps over $X$ to those over the curves.
For example, to treat $\beta_X$, we take curves $C_1,\cdots, C_r$
in $X$ such that the kernel of $H^2(k, \Pic(X^s)) \to \oplus_i H^2(k, \Pic(C_i^s))$ has a finite exponent. This is guaranteed by the property the morphism $\uPic_{X/k} \to \prod_i \uPic_{C_i/k}$ of the Picard functors is a direct summand up to isogeny, which is in turn reduced to similar properties for the connected components and the component groups. 
Then the later properties are solved by the Hodge index theorem for divisors and the weak Lefschetz theorem for the first \'etale cohomology group.

\subsection{Notations and conventions}

We take the following notations and conventions throughout this paper.

\subsubsection*{Fields}

By a \emph{finitely generated field}, we mean a field which is finitely generated over a prime field. 

For any field $k$, denote by $k^s$ (resp. $\bar k$) the separable closure (resp. algebraic closure). Denote by $G_k=\Gal(k^s/k)$ the absolute Galois group of $k$. 

\subsubsection*{Varieties}

By a \emph{variety} over a field $k$, we mean a scheme 
which is separated, \emph{geometrically integral} and of finite type over $k$. By a \emph{curve}, we mean a variety of dimension one. 

For a variety $X$ over a field $k$, we usually denote by $X^s=X_{k^s}$ the base change to the separable closure if the base field $k$ is clear from the context.

\subsubsection*{Cohomology}

The default sheaves and cohomology over varieties are with respect to the small \'etale site.

\subsubsection*{Brauer groups}

For any noetherian scheme $X$, denote the \emph{cohomological Brauer group}
$$\Br(X):=H^2(X, \GG_m)_\tor,$$
which is the torsion part of the \'etale cohomology group. 
Note that if $X$ is regular, which is always the case in this paper, the group $H^2(X, \GG_m)$ is automatically torsion.

\subsubsection*{Abelian groups}
For any abelian group $G$, integer $m$ and prime $\ell$, we introduce the following notations. 
\begin{itemize}
\item $[m]:G\to G$ is the homomorphism given by multiplication by $m$.
\item $G[m]$ is the kernel of $[m]:G\to G$. 
\item $G_\tor$ is the union of $G[m]$ in $G$ over all positive integers $m$.
% \item $G_{\rm free}=G/G_\tor$ is the free part of $G$. 
\item $G[\ell^\infty]$ is the union of $G[\ell^n]$ in $G$ over all $n\geq1$.

% \item If $p$ is a prime,  $G\nonp$ denotes the union of $G[m]$ in $G$ over all integer $m$ coprime to $p$; if $p=0$, set  $G\nonp=G_\tor$. This notation always applies to the situation that $p$ denotes the characteristic of a base field.

\item $\displaystyle T_\ell(G)=\Hom_{\ZZ}(\QQ_\ell/\ZZ_\ell, G)
=\varprojlim_n G[\ell^n]$, where the transition map of the inverse system is given by $[\ell]:G[\ell^{n+1}]\to G[\ell^n]$.

\item $\displaystyle V_\ell(G)=T_\ell(G)\otimes_{\ZZ_\ell}\QQ_\ell$.

% \item $G_{\rm div}=\cup_{\phi\in \Hom_{\ZZ}(\QQ, G)}\phi(\QQ)$ denotes the subgroup of (successively) divisible elements of $G$.
% \item $G_{\rm ndiv}=G/G_{\rm div}$ is the non-divisible part of $G$. 
\end{itemize}

\subsubsection*{Acknowledgment}

The author would like to thank Marco D'Addezio for informing him the paper of Colliot-Th\'el\`ene and Skorobogatov.
The author thanks Yanshuai Qin for pointing out a mistake in an old version of the paper. 
The author would like to express his deep thanks to the Institute for Advanced Study at Tsinghua University.
Part of the paper was written when the author visited the institute in the academic year 2018-2019.

The author is supported by the grant RTG/DMS-1646385 from the National Science Foundation of the USA.

\section{Tate conjecture and geometric Brauer groups}
\label{sect tate brauer}

The goal of this section is to give a complete proof of Theorem \ref{equivalence}.

\subsection{Preliminary results} \label{sec prelim}

In this section, we review some standard results on Picard groups and the Kummer sequence for the \'etale cohomology.

\subsubsection*{Review on Picard groups}
%\label{sect picard}

Here we review some definitions and basic properties about the Picard group, the Neron--Severi group, and the Picard functor.

Let $X$ be a projective scheme over a field $k$. 
Denote by $\Pic^0(X)$ the subgroup of $\Pic(X)$ of algebraically trivial line bundles (cf. \cite[Definition 9.5.9]{Kle}), and define $\NS(X)$ by the exact sequence 
$$
0\lra \Pic^0(X) \lra \Pic(X)\lra  \NS(X)\lra 0.
$$
Note that a line bundle over $X$ is algebraically trivial if and only if it is algebraically trivial over $X^s$.
Therefore, 
$$
\NS(X) =\Im(\Pic(X)\to  \NS(X^s)).
$$

By \cite[$n^\circ$232, \S6]{FGA} or \cite[\S8.2, Thm. 3]{BLR}, 
%By \cite[Theorem 9.4.8]{Kle}, 
the Picard functor $\uPic_{X/k}$ is represented by a group scheme, locally of finite type over $k$. 
Denote by $\uPic^0_{X/k}$ the identity component of (the group scheme representing) $\uPic_{X/k}$.
By \cite[Lem 9.5.1]{Kle}, $\uPic^0_{X/k}$ is a group scheme of finite type over $k$, open and closed in $\uPic_{X/k}$.
If $X$ is geometrically normal, by \cite[Prop. 9.5.3, Thm. 9.5.4]{Kle},  $\uPic^0_{X/k}$ is actually projective over $k$.
In this case, by \cite[$n^\circ$236-16, Cor. 3.2]{FGA}, the reduced structure 
$\uPic^0_{X/k,\red}$ of $\uPic^0_{X/k}$ is an abelian variety over $k$.

There are canonical injections
$$\Pic(X)\lra \uPic_{X/k}(k), \quad \Pic^0(X)\lra \uPic^0_{X/k}(k).$$
They are isomorphisms if $X(k)$ is non-empty
or $k$ is separably closed.
See \cite[\S 8.1, Prop. 4]{BLR} and \cite[Prop 9.5.10, Thm 9.2.5]{Kle}.

By \cite[Exp. XIII, Theorem 5.1]{SGA6}, $\NS(X^s)$ is a finitely generated abelian group. Then $\NS(X)$ is also finitely generated. 

%Consequently, if $k$ is finite, then $\Pic^0(X)$ is finite and $\Pic(X)$ is finitely generated. 

\subsubsection*{The Kummer sequence}

We review the following standard results, and provide a proof for the sake of readers. 

\begin{lem} \label{kummer}
Let $X$ be a projective and smooth variety over a separably closed field $k$ of characteristic $p\geq 0$, and $\ell\neq p$ be a prime number. 
Then the following holds:
\begin{itemize}
\item[(1)] $\Br(X)[\ell]$ is finite.
\item[(2)] There is a canonical exact sequence  
$$
0 \lra \NS(X) \otimes_\ZZ \ZZ_\ell \lra 
H^2(X, \ZZ_\ell(1)) \lra T_\ell \Br(X) \lra 0.
$$
\end{itemize}
\end{lem}

\begin{proof}

These are consequence of the Kummer sequence
$$
0\lra \Pic(X)/\ell^n \lra H^2(X, \muln) \lra \Br(X)\belln \lra 0.
$$
Note that $H^2(X, \muln)$ is finite, so $\Br(X)\belln$ is finite for any $n$.
% This proves (1) by combining the group-theoretical results recalled in Fact \ref{structure}. 

For the second statement, the morphism $[\ell]:\Pic^0 _{X/ k}\to \Pic^0 _{X/ k}$ is finite, \'etale and surjective, and thus it is surjective on $k$-points (as $k$ is separably closed). 
In other words, $[\ell]:\Pic^0(X)\to \Pic^0(X)$ is surjective.
This implies an isomorphism $\Pic(X)/\ell^n \to \NS(X)/\ell^n$. 
Then the exact sequence becomes
$$
0\lra \NS(X)/\ell^n \lra H^2(X, \muln) \lra \Br(X)\belln \lra 0.
$$
Taking inverse limit, we give the exact sequence in (2). 
\end{proof}

The first part of the lemma is often combined with the following basic results in group theory.

\begin{fact} \label{structure}
Let $\ell$ be a prime and $G$ be an abelian $\ell$-group which is co-finite in the sense  that $G[\ell]$ is finite. Then the following holds:
\begin{itemize}
\item[(1)] $G$ is isomorphic to $(\QQ_\ell/\ZZ_\ell)^r\oplus G_0$ for some (finite) integer $r\geq0$ and some finite $\ell$-group $G_0$. See \cite[Thm. 3, Thm. 4, Thm. 9]{Kap}. 
\item[(2)] $T_\ell G$ is a free $\zl$-module of finite rank. 
\item[(3)] If $G\to H$ is a surjective homomorphism, and $H$ has a finite exponent, then 
$H$ is a finite group. 
\end{itemize}
\end{fact}

\subsection{The Galois invariants}

Now we prove Theorem \ref{equivalence}.
Let $(X,k,\ell)$ be as in the theorem. 
By Lemma \ref{kummer}, we have an exact sequence
$$
0 \lra \NS(X^s)_{\QQ_\ell} \lra 
H^2(X^s, \QQ_\ell(1)) \lra V_\ell \Br(X^s) \lra 0.
$$
We first prove that the sequence has a $G_k$-equivariant splitting. 
It suffices to treat the case $\dim X>1$.

Fix an ample line bundle $L$ over $X$. 
Consider the intersection pairing 
$$
\NS(X^s)_{\QQ_\ell} \times \NS(X^s)_{\QQ_\ell} \lra \QQ_\ell,\quad
(D,E)\longmapsto D\cdot E\cdot L^{\dim X-2},
$$
and the cup-product pairing 
$$
H^2(X^s, \QQ_\ell(1))\times H^2(X^s, \QQ_\ell(1)) \lra \ql,\quad
(\alpha,\beta)\longmapsto \alpha\cup\beta\cup c_1(L)^{\dim X-2}.
$$
The pairings are compatible. 
By \cite[XIII, Thm. 4.6(i)(vi)]{SGA6}, the intersection pairing is non-degenerate. 
Therefore, we can take the orthogonal complement of 
$\NS(X^s) _{\QQ_\ell} $ in $H^2(X^s, \QQ_\ell(1))$
to get a splitting of the sequence.
Note that the cup-product pairing is also non-degenerate by the hard Lefschetz theorem of Deligne \cite[Thm. 4.1.1]{Del}, but the splitting only needs the non-degeneracy of the intersection pairing.

Once the exact sequence is split as $G_k$-modules, taking 
$G_k$-invariants gives an exact sequence
$$
0 \lra \NS(X^s)^\gk \otimes_\ZZ \QQ_\ell \lra 
H^2(X^s, \QQ_\ell(1))^\gk \lra (V_\ell \Br(X^s))^\gk \lra 
0.$$
It suffices to prove both of the natural maps 
$$
\NS(X) \lra \NS(X^s)^{G_k}, \quad\
T_\ell (\Br(X^s)^\gk) \lra (T_\ell \Br(X^s))^\gk
$$
have finite kernels and cokernels. 

We first treat the map
$$
\NS(X) \lra \NS(X^s)^{G_k}.
$$
It is injective by definition. For its cokernel, take $G_k$-invariants of the exact sequence
$$
0\lra \Pic^0(X^s) \lra \Pic(X^s)\lra  \NS(X^s)\lra 0.
$$
We have an exact sequence
$$
\Pic(X^s)^\gk \lra  \NS(X^s)^\gk \lra H^1(G_k, \Pic^0(X^s)).
$$
The last arrow has a finite image, since 
$\NS(X^s)$ is finitely generated and $H^1(G_k, \Pic^0(X^s))$ is torsion. 
Then it suffices to prove that 
$\Pic(X) \to \Pic(X^s)^\gk$
has a torsion cokernel.
The Hochschild--Serre spectral sequence 
$$H^a(G_k, H^b(X^s, \GG_m))
\Longrightarrow H^{a+b}(X, \GG_m)$$
induces an exact sequence
$$\Pic(X)\lra \Pic(X^s)^\gk \lra \Br(k)\lra \Br(X).$$
The cokernel of the first map is torsion, since $\Br(k)$ is torsion. 
This treats the first map.

The second map 
$$
T_\ell (\Br(X^s)^\gk) \lra (T_\ell \Br(X^s))^\gk
$$
is actually an isomorphism. 
In fact, it is injective by taking the Tate module of the inclusion $\Br(X^s)^\gk \to \Br(X^s)$. For the surjectivity, denote 
$$M:=T_\ell \Br(X^s)=T_\ell (\Br(X^s)[\ell^\infty]_\div).$$ 
Then there is a canonical isomorphism 
$$M\otimes (\ql/\zl) \lra \Br(X^s)[\ell^\infty]_\div.$$
Since $M$ is a free $\ZZ_\ell$-module of finite rank (cf. Lemma \ref{kummer} and Fact \ref{structure}), $M^\gk$ is saturated in $M$; i.e., $M/M^\gk$ is torsion-free. It follows that $M^\gk$ is a direct summand of $M$ as $\ZZ_\ell$-modules. As a consequence, 
$$M^\gk \otimes (\ql/\zl) \lra M \otimes (\ql/\zl)$$
is injective.
Thus we have an injection
$$M^\gk\otimes (\ql/\zl) \lra \Br(X^s)[\ell^\infty]_\div,$$
which induces an injection
$$M^\gk \otimes (\ql/\zl) \lra \Br(X^s)^\gk[\ell^\infty].$$
Taking Tate modules, this gives an injection 
$$M^\gk \lra T_\ell(\Br(X^s)^\gk).$$
This gives an inverse of the original map. 
The proof of Theorem \ref{equivalence} is complete.

\section{Comparison of the Brauer groups}

The goal of this section is to prove Theorem \ref{comparison}.
Start with the Hochschild--Serre spectral sequence
$$
E_2^{a,b}= H^a(k, H^b(X^s, \GG_m))
\Longrightarrow  H^{a+b}(X, \GG_m).
$$
It gives a filtration of $\Br(X)$ whose successive quotients are
$$
E_\infty^{2,0}=\mathrm{coker}(\Pic(X^s)^{G_k}\to \Br(k)),
$$
$$
E_\infty^{1,1}=\ker(\alpha_X:H^1(k, \Pic(X^s))\to H^3(k,\GG_m)),
$$
$$
E_\infty^{0,2}=\ker(\gamma_X: E_3^{0,2}\to E_3^{3,0}),
$$
where 
$$
E_3^{0,2}=\ker(\beta_X: \Br(X^s)^{G_k}\to H^2(k, \Pic(X^s)),
$$
$$
E_3^{3,0}=\coker(\alpha_X:H^1(k, \Pic(X^s))\to H^3(k,\GG_m)).
$$
This gives a canonical long exact sequence
\begin{multline*}
0\lra \ker(\alpha_X:H^1(k, \Pic(X^s))\to H^3(k,\GG_m)) \\
\lra \Br(X)/\Br(k) \stackrel{ }{\lra} 
E_3^{0,2}\stackrel{\gamma_X}{\lra} E_3^{3,0}
\lra H^3(X,\GG_m).
\end{multline*}
Here the last arrow is added because of 
$$
E_\infty^{3,0}=E_4^{3,0}=\coker(\gamma_X:E_3^{0,2}\to E_3^{3,0}).
$$
The key is the following result on the connecting maps.

\begin{lem} \label{connecting map}
The images of all the canonical maps 
$$\alpha_X:H^1(k, \Pic(X^s))\lra H^3(k,\GG_m),$$
$$\beta_X:\Br(X^s)^{G_k} \lra H^2(k, \Pic(X^s)),$$ 
and
$$\gamma_X:E_3^{0,2}\lra E_3^{3,0}$$ 
have finite exponents.  
\end{lem}

Now the lemma implies Theorem \ref{comparison}(1)(2)(3) immediately. 
In fact, apply Lemma \ref{connecting map} to the long exact sequence. 
Theorem \ref{comparison}(1) is immediately obtained by the lemma for $\alpha_X$.
On the other hand, by the lemma for $\gamma_X$, the canonical map 
$$\Br(X)/\Br(k) \stackrel{ }{\lra} 
E_3^{0,2}$$
has a cokernel of finite exponent.
By the lemma for $\beta_X$, the cokernel of $E_3^{0,2}\to \Br(X^s)^{G_k}$ has a finite exponent. 
Then the cokernel of
$$\Br(X)/\Br(k) \stackrel{ }{\lra} 
\Br(X^s)^{G_k}$$
also has a finite exponent.
Then Theorem \ref{comparison}(2)(3) follows from
Lemma \ref{kummer} and Fact \ref{structure}(3).

\subsection{Extra arguments over finite fields}

The goal here is to deduce Theorem \ref{comparison}(4), which asserts that if $k$ is a finite field, then the canonical map $\Br(X) \to \Br(X^s)^{G_k}$
has a finite kernel and a finite cokernel.
We first list the following well-known facts for Galois cohomology of finite fields.

\begin{fact} \label{cohomology finite field}
Let $k$ be a finite field, and $M$ be a discrete $G_k$-module.
\begin{enumerate}[(1)]
\item $H^i(G_k,M)=0$ if $i\geq2$ and $M$ is a torsion abelian group. This follows from the fact that the cohomological dimension of $k$ is 1. See \cite[II,\S3.3,(a)]{Ser} for example. 
\item $H^i(k,\GG_m)=0$ for any $i\geq2$. This is an example of the above result.  
\item $H^1(G_k,M)$ is finite if $M$ is a finitely generated abelian group.
In fact, since $H^1(G_k,M)$ is torsion, it suffices to check that it is finitely generated. 
This is easily seen by the crossed homomorphisms, since $G_k$ is topologically generated by one element. 
\item $H^2(G_k,M)[n]$ is finite for any positive integer $n$ if $M$ is an abelian group with a finitely generated free part $M/M_\tor$. In fact, taking cohomology of $0\to M_\tor\to M \to M/M_\tor\to 0$, we can assume that $M$ is a free $\ZZ$-module of finite rank by (1). Taking cohomology of $0\to M \stackrel{n}{\to} M \to M/n \to 0$, the result follows by (3). 
\end{enumerate}
\end{fact}

Return to the proof of Theorem \ref{comparison}(4).
By $H^2(k,\GG_m)=H^3(k,\GG_m)=0$, the exact sequence associate to the spectral sequence simplifies as
$$
0\lra H^1(k, \Pic(X^s))
\lra \Br(X) \stackrel{ }{\lra} 
\Br(X^s)^{G_k}\stackrel{\beta_X}{\lra} H^2(k, \Pic(X^s)).
$$
It suffices to prove that $H^1(k, \Pic(X^s))$ and $\Im(\beta_X)$ are both finite. 

We first prove that $\Im(\beta_X)$ is finite.
By Lemma \ref{connecting map}(2), $\Im(\beta_X)$ has a finite exponent.
Note that it is a subgroup of $H^2(k, \Pic(X^s))$.
We claim that the free part of $\Pic(X^s)$ is finitely generated, so that we can apply Fact \ref{cohomology finite field}(4) to get the finiteness. 
Consider the exact sequence
$$
0\lra \Pic^0(X^s) \lra \Pic(X^s)\lra \NS(X^s)\lra 0. 
$$
As recalled in \S\ref{sec prelim}, $\NS(X^s)$ is finitely generated. Moreover, $\Pic^0(X^s)$ is torsion, since it is the $k^s$-point of the abelian variety $\underline{\Pic}^0_{X/k,\red}$.
This proves that the image of $\beta_X$ is finite. 

Now we prove that $H^1(k, \Pic(X^s))$ is finite. 
By the exact sequence 
$$
0\lra \Pic^0(X^s) \lra \Pic(X^s)\lra \NS(X^s)\lra 0,
$$
we have an exact sequence
$$H^1(k, \Pic^0(X^s))\lra H^1(k, \Pic(X^s))\lra H^1(k, \NS(X^s)).$$
The key is the vanishing theorem $H^1(k, \Pic^0(X^s))=0$
by Lang \cite[Thm. 2]{Lan} applied to the abelian variety $\underline{\Pic}^0_{X/k,\red}$. 
As $\NS(X^s)$ is finitely generated, $H^1(k, \NS(X^s))$ is finite by Fact \ref{cohomology finite field}(2). 
This proves Theorem \ref{comparison}(4).

\subsection{The images of $\alpha_X$ and $\gamma_X$}

We first prove Lemma \ref{connecting map} for $\alpha_X$ and $\gamma_X$; i.e., the images of
$$\alpha_X:H^1(k, \Pic(X^s))\lra H^3(k,\GG_m)$$
and
$$\gamma_X:E_3^{0,2}\lra E_3^{3,0}$$
have finite exponents, where
$$
E_3^{0,2}=\ker(\beta_X: \Br(X^s)^{G_k}\to H^2(k, \Pic(X^s)),
$$
$$
E_3^{3,0}=\coker(\alpha_X:H^1(k, \Pic(X^s))\to H^3(k,\GG_m)).
$$

Both results are easy consequences of the fact that the canonical map 
$$H^3(k,\GG_m)\lra H^3(X,\GG_m)$$
has a finite exponent. 
We first check the fact by a standard argument using a multi-section. 
In fact, let $P\in X$ be a closed point whose residue field $k'$ is finite and separable over $k$. 
It suffices to prove that the kernel of the composition
$$
H^3(k,\GG_m)\lra H^3(X,\GG_m) \lra H^3(k',\GG_m)
$$
has a finite exponent. 
It is known that the kernel of 
$$
\Res_{k'/k}:H^3(k, \GG_m) \lra H^3(k', \GG_m)
$$
has exponent $[k':k]$, as its composition with 
$$
\mathrm{Cores}_{k'/k}:H^3(k', \GG_m) \lra H^3(k, \GG_m)
$$
gives $\mathrm{Cores}_{k'/k}\circ \Res_{k'/k}=[k':k]$. 
Here $\mathrm{Cores}_{k'/k}$ is the co-restriction map in group cohomology (cf. \cite[Cor. I.5.7]{NSW}).

Now we prove that the image of $\beta_X$ has a finite exponent. 
Recall that the spectral sequence gives a canonical exact sequence
$$
E_3^{0,2}\stackrel{\beta_X}{\lra} E_3^{3,0}\lra H^3(X,\GG_m),
$$
where
$$
E_3^{3,0}=\coker(\alpha_X:H^1(k, \Pic(X^s))\to H^3(k,\GG_m)).
$$
It suffices to prove that the kernel of $E_3^{3,0}\to H^3(X,\GG_m)$
has a finite exponent. 
This is implied by the fact that the canonical composition 
$$H^3(k,\GG_m)\lra E_3^{3,0}\lra H^3(X,\GG_m)$$ 
has a kernel of finite exponent.
Thus the image of $\beta_X$ has a finite exponent. 

On the other hand, in the composition, $\Im(\alpha_X)$ is contained in the kernel of $H^3(k,\GG_m)\to H^3(X,\GG_m)$. So it also has a finite exponent.

\subsection{The image of $\beta_X$}

Now we prove the remaining part of Lemma \ref{connecting map}, which asserts that 
the image of  
$$\beta_X:\Br(X^s)^{G_k} \lra H^2(k, \Pic(X^s))$$ 
has a finite exponent.

As mentioned in the introduction, the key idea is a pull-back trick of 
Colliot-Th\'el\`ene and Skorobogatov \cite{CTS1}.
It is based on the observation that if $\dim X=1$, then $\Br(X^s)=0$ and thus the connecting map $\beta_X$ is zero. 
For general $\dim X$, we take suitable curves in $X$ and apply the functoriality of the connecting maps to relate the maps over $X$ to those over the curves.

Let $k'$ be a finite separable extension of $k$. We claim that the result for $X_{k'}$ implies that for $X$. 
In fact, by the commutative diagram
$$
\xymatrix{
&
\quad \Br(X_{k'^s})^{G_{k'}} \quad  \ar[r]^{\beta_{X_{k'}}\ }
&
\quad H^2(k', \Pic(X_{k'^s})) \quad 
\\
&
\quad \Br(X_{k^s})^{G_k} \quad \ar[r]^{\beta_{X}\ } \ar[u]
&
\quad H^2(k, \Pic(X_{k^s})) \quad \ar[u]
}
$$
it suffices to prove that the kernel of the restriction map
$$
\Res_{k'/k}:H^2(k, \Pic(X_{k^s})) \lra H^2(k', \Pic(X_{k'^s}))
$$
has a finite exponent. 
Using the corestriction map, this is similar to the result about $H^3(k, \GG_m)\to H^3(k', \GG_m)$ above.

As a consequence of the claim, we can replace $k$ by any finite separable extension $k'$ in the proof.
In particular, we can assume that the natural map 
$$\NS(X) \lra \NS(X^s)$$
is an isomorphism. 
% Moreover, we can assume that $\Pic^0 _{X/k, \mathrm{red}}$ is geometrically reduced, 
% so  it is an abelian variety over $k$.

Denote $r=\rank\, \NS(X)$.
Take very ample line bundles $L,L_1,\cdots, L_r$ over $X$, such that the images of $L_1,\cdots, L_r$ in $\NS(X)$ are linearly independent over $\ZZ$. 
For $i=1,\cdots, r$, let $C_i$ be a projective and smooth curve in $X$ obtained as successive hyperplane sections associated to powers of line bundles $L,L,\cdots, L, L_i$. 
Here $L$ is used $n-2$ times for $n=\dim X$. As above, the existence of $C_i$ follows from Bertini's theorem if $k$ is infinite. 
If $k$ is finite, Bertini's theorem works by replacing $k$ by a suitable finite extension, or
we can apply the result of Gabber \cite{Gab} or Poonen \cite{Poo} without changing $k$.
By construction, $C_i$ is numerically equivalent to $m_i\, L^{n-2}\cdot L_i$
for some positive integer $m_i$.

Consider the diagram:
$$
\xymatrix{
&
\quad \Br(X^s)^{G_k} \quad  \ar[d] \ar[r]^{\beta_X\quad}
&
\quad H^2(k, \Pic(X^s)) \quad \ar[d]_{}
\\
&
\quad \oplus_i \Br(C^s_i)^{G_k} \quad \ar[r]^{\oplus_i\beta_{C_i}\quad}
&
\quad \oplus_i H^2(k, \Pic(C^s_i)) \quad
}
$$
Here the vertical arrows are given by pull-backs via the immersion $C_i\to X$.
The diagram is commutative by functoriality of the spectral sequence. 
By \cite[Cor. 5.8]{Gro3}, $\Br(C^s_i)=0$. 
Then it suffices to prove that 
the kernel of the canonical map 
$$
H^2(k, \Pic(X^s)) \lra \oplus_i H^2(k, \Pic(C^s_i))
$$
has a finite exponent for general $k$ and is finite for finite $k$.

The above map is induced by the map 
$$
\Pic(X^s) \lra \oplus_i \Pic(C^s_i),
$$
which is actually the $k^s$-points of the morphism
$$
\tau: \uPic_{X/k,\red} \lra \prod_i \uPic_{C_i/k}
$$
of the Picard schemes.
We claim that there is an abelian subvariety $A$ of $\prod_i \uPic_{C_i/k}$ such that the induced morphism 
$$
\tau':\uPic_{X/k,\red}\times A \lra \prod_i \uPic_{C_i/k}
$$
has a finite kernel and a finite cokernel (as group schemes over $k$). 

Assuming the existence of $A$. Then $\tau'$ gives two short exact sequences by taking its kernel, image and cokernel. 
Taking cohomology, we see that the kernel of the canonical map 
$$
H^2(k, \Pic(X^s))\oplus H^2(k,A) \lra \oplus_i H^2(k, \Pic(C^s_i))
$$
has a finite exponent for general $k$. Then the same finiteness property holds for the kernel of the canonical map 
$$
H^2(k, \Pic(X^s)) \lra \oplus_i H^2(k, \Pic(C^s_i)).
$$

For the existence of $A$, we first claim that the canonical map 
$$
\phi:\NS(X^s) \lra  \oplus_{i=1}^r \NS(C_i^s),
$$
which is the map between the component groups induced by $\tau$,
has a finite kernel and a finite cokernel. 
In fact, there is a canonical isomorphism $\deg:\NS(C_i^s)\to \ZZ$.
Under this isomorphism, the above map is given by  
$$
\alpha\longmapsto (m_1\, L^{n-2}\cdot L_1\cdot \alpha,\cdots, m_r\, L^{n-2}\cdot L_r\cdot \alpha).
$$
Then it has a finite kernel by the non-degeneracy of
the intersection pairing 
$$
\NS(X^s)_{\QQ} \times \NS(X^s)_{\QQ} \lra \QQ,\quad
(D,E)\longmapsto D\cdot E\cdot L^{n-2}.
$$
See \cite[XIII, Thm. 4.6(i)(vi)]{SGA6}. 
Then $\phi$ has a finite cokernel by comparing the ranks of its source and its target. 

By the claim, to prove the existence of $A$, it suffices to prove that there is an abelian subvariety $B$ of $\prod_i \uPic^0_{C_i/k}$ such that the induced morphism 
$$
\tau'':\uPic^0_{X/k,\red}\times B \lra \prod_i \uPic^0_{C_i/k}
$$
has a finite kernel and a finite cokernel. 
In fact, if $B$ exists, we can simply set $A=B$. 

Finally, we prove the existence of $B$.
Note that the morphism 
$$
\uPic^0_{X/k,\red} \lra \prod_i \uPic^0_{C_i/k}
$$
is a homomorphism of abelian varieties.
By Poincare's complete reducibility theorem, it suffices to prove that the kernel of the morphism is finite. 
Then it suffices to prove that the kernel of the morphism 
$\uPic^0_{X/k,\red} \to \uPic^0_{C_1/k}$ is finite. 
By the weak Lefschetz theorem (cf. \cite[I, Cor. 9.4]{FK}), the restriction map 
$H^1(X^s,\ZZ_\ell(1)) \to H^1(C_1^s,\ZZ_\ell(1))$ is injective.
The restriction map is isomorphic to the map $T_\ell \uPic^0_{X/k,\red} \to T_\ell\uPic^0_{C_1/k}$
of the Tate modules. 
Hence, the kernel of  
$\uPic^0_{X/k,\red} \to \uPic^0_{C_1/k}$ is finite. 
This finishes the proof.

%------------------------------------------------------
\end{document}